\documentclass[reqno]{amsart}

\usepackage{mathtools}
\usepackage{float}
\usepackage{amsmath}
\usepackage{amsfonts}
\usepackage{amssymb}
\usepackage{graphicx}
\usepackage{hyperref}
\usepackage{mathrsfs}
\usepackage{pb-diagram}
\usepackage{epstopdf}
\usepackage{amsmath,amsfonts,amsthm,enumerate,amscd,latexsym,curves}
\usepackage{bbm}
\usepackage[mathscr]{eucal}
\usepackage{epsfig,epsf,xypic,epic}
\usepackage{caption}
\usepackage{subcaption}
\usepackage{tikz}
\usepackage{geometry}
\usepackage{enumitem}
\usetikzlibrary{matrix,arrows,decorations.pathmorphing}
\usetikzlibrary{cd}
\usepackage{tikz-cd}

\usetikzlibrary{matrix,arrows,decorations.pathmorphing}

\theoremstyle{theorem}
\newtheorem{thm}{Theorem}[section]

\newtheorem{lem}[thm]{Lemma}
\newtheorem{cor}[thm]{Corollary}
\newtheorem{prop}[thm]{Proposition}

\theoremstyle{definition}
\newtheorem{defn}[thm]{Definition}
\newtheorem{ques}[thm]{Question}

\newtheorem{rem}[thm]{Remark}

\newtheorem{hypo}[thm]{Hypothesis}

\numberwithin{equation}{section}

\newcommand{\Hom}{\mathrm{Hom}}

\newcommand{\bb}[1]{\mathbb{#1}}
\newcommand{\cu}[1]{\mathcal{#1}}
\newcommand{\til}[1]{\widetilde{#1}}

\newcommand{\tildeL}{\widetilde{L}}
\newcommand{\be}{\begin{equation}}
\newcommand{\ee}{\end{equation}}

\newcommand{\del}{\partial}

\def\bZ{\mathbb{Z}}

\def\bR{\mathbb{R}}
\def\bC{\mathbb{C}}
\def\bP{\mathbb{P}}
\def\bL{\mathbb{L}}

\def\cA{\mathcal{A}}
\def\cM{\mathcal{M}}
\def\cT{\mathcal{T}}
\def\cY{\mathcal{Y}}
\def\cL{\mathcal{L}}
\def\id{\text{\rm id}}

\def\vol{\text{\rm vol}}
\def\ev{\text{\rm ev}}
\def\Hom{\text{\rm Hom}}
\def\Self{\text{\rm Self}}
\def\Aut{\operatorname{Aut}}
\def\Fuk{\mathfrak{Fuk}}
\def\Ev{\operatorname{Ev}}
\def\ev{\operatorname{ev}}
\def\trop{\operatorname{trop}}
\def\PO{\mathfrak{PO}}

\begin{document}
	
	\title[Microlocalization]{Fukaya's immersed Lagrangian Floer theory and
microlocalization of Fukaya category}

	\author[Oh]{Yong-Geun Oh}
\address{Center for Geometry and Physics, Institute for Basic Science (IBS),
79 Jigok-ro 127beon-gil, Nam-gu, Pohang, Gyeongsangbuk-do, Korea 37673
\& POSTECH, 77 Cheongam-ro, Nam-gu, Pohang, Gyeongsangbuk-do, Korea 37673}
\email{yongoh1@postech.ac.kr}

	\author[Suen]{Yat-Hin Suen}
\address{School of Mathematics, Korea Institute for Advanced Study,
    85 Hoegiro Dongdaemun-gu, Seoul, Korea 02455}
	\email{yhsuen@kias.re.kr}

	\date{\today}

\begin{abstract} Let $\Fuk(T^*M)$ be the Fukaya category 
in the Fukaya's immersed Lagrangian Floer theory \cite{fukaya:immersed} which is
generated by immersed Lagrangian submanifolds \emph{with clean self-intersections}.
This category is monoidal in that the product of two such immersed Lagrangian submanifolds 
remains to be a Lagrangian immersion with clean self-intersection. Utilizing this monoidality of
Fukaya's immersed Lagrangian Floer theory,
we prove the following generation result in this Fukaya category of the cotangent bundle,
which is  the counterpart of Nadler's generation result \cite{Nadler}
for the Fukaya category generated by the exact embedded Lagrangian branes. More specifically we
prove that for a given triangulation $\cT = \{\tau_{\mathfrak a}\}$ fine enough the Yoneda module
$$
\mathcal{Y}_{\bL}: = hom_{\Fuk(T^*M)}(\cdot, \bL)
$$
can be expressed as a twisted complex with terms 
$hom_{\Fuk(T^*M)}(\alpha_M(\cdot), L_{\tau_{\mathfrak a}*})$
for any curvature-free (aka tatologically unostructed) object $\bL$. Using this, we also extend Nadler's equivalence theorem
between the dg category $Sh_c(M)$ of constructible sheaves on $M$ and the triangulated envelope of
$\Fuk^0(T^*M)$ to the one over the Novikov field $\mathbb K$.

	\end{abstract}
\thanks{The first named author is supported by the IBS project \# IBS-R003-D1.}

	\maketitle
	
	\tableofcontents

\section{Introduction}
The notion of tropical Lagrangian multi-sections over a rational fan $\Sigma$ was introduced by Payne in 
\cite{branched_cover_fan}
(See also \cite{Suen_trop_lag}.)  Each tropical Lagrangian multi-section $\bb{L}^{\trop}$ 
determines a Lagrangian subset $\Lambda_{\bb{L}^{\trop}}\subset \Lambda_{\Sigma}\subset T^*M_{\bb{R}}$, where $M_{\bb{R}}$ is dual to the vector sapce where $\Sigma$ lives. 
The following question is of  fundamental interest in the framework of 
 the SYZ mirror symmetry.
	
\begin{ques}[Lagrangian realization problem]\label{ques:existence}
Given a tropical Lagrangian multi-section $\bb{L}^{\trop}$ over a complete fan $\Sigma$, 
is there a  Lagrangian multi-section $\bb{L}$ as an object in the tautologically unobstructed Fukaya category
$\Fuk^0(Y,\Lambda_{\Sigma})$  such that 
$\bb{L}^{\infty}\subset\Lambda_{\bb{L}^{\trop}}^{\infty}$?
\end{ques}
(We refer to Definition \ref{defn:strictification}, \eqref{eq:L-decompose} and Proposition \ref{prop:curvature-free}
in Appendix for the definition of 
$\Fuk^0(Y,\Lambda_{\Sigma})$ which corresponds to $\Fuk_0$ in Proposition \ref{prop:curvature-free}.)
	
 If $\bb{L}^{\trop}$ can be realized by a Lagrangian multi-section $\bb{L}$ in the sense of 
 \cite[Question 1.6]{oh-suen:lag-multi-section}, 
 its mirror is a toric vector bundle $\cu{E}_{\bb{L}}$ whose associated tropical Lagrangian multi-section $\bb{L}_{\cu{E}_{\bb{L}}}^{\trop}$ is nothing but $\bb{L}^{\trop}$. However, in \cite{Suen_trop_lag}, the second-named author gave an example of a 2-fold tropical Lagrangian multi-section over the fan of $\bb{P}^2$ that \emph{does not} arise from toric vector bundles on $\bb{P}^2$. By mirror symmetry, we should not expect it can be realized by an unobstructed Lagrangian submanifold in $Y$. Hence there must be some extra assumptions that need to be put on $\bb{L}^{\trop}$ in order to get an affirmative answer to the realization problem. 
 In \cite{oh-suen:lag-multi-section}, we  introduce the notion of the \emph{$N$-generic condition} 
 \cite[Definition 5.3]{oh-suen:lag-multi-section} and prove the following.
	
\begin{thm}[Theorem 5.15 \cite{oh-suen:lag-multi-section}]
\label{thm:unobs_immersed_Lag-intro}
			Let $\bb{L}^{\trop}$ be a $N$-generic 2-fold tropical Lagrangian multi-section over a complete 2-dimensional fan $\Sigma$ with $N\geq 3$. Then there is a spin, graded and immersed 2-fold Lagrangian multi-section $\bb{L}$ in $Y$, whose immersed sector is concentrated at degree 1 and $\bb{L}^{\infty}\subset\Lambda_{\bb{L}^{\trop}}^{\infty}$. In particular, $\bb{L}$ is tautologically unobstructed. When $\bb{L}$ is embedded, the topology of the underlying surface has Betti numbers $b_0(\bb{L})=1,\,b_1(\bb{L})=N-3,\,b_2(\bb{L})=0$.
	\end{thm}

To ensure the Lagrangian immersion constructed in Theorem 1.2 admits a mirror object in $\cu{P}erf_T(X_{\Sigma})$, 
the Fang-Liu-Trenman-Zaslow \emph{coherent-constructible correspondence} 
and the Nadler-Zaslow's
\emph{microlocalization
functor} play an important role. In the proof of this equivalence statement, the Lagrangian correspondence and
the representation of Yoneda module $\cY(L)$ of a Lagrangian brane in $T^*M$ by a twisted complex of standard
branes of simplicies of triangulation of $M$ play a crucial role. Moreover in the latter representation,
the study of K\"unneth-type functors is essential. 

Therefore one would like the considered Fukaya category to be \emph{monoidal} so that taking 
the tensor product makes sense therein. However the Fukaya category of immersed Lagrangians cannot be monoidal
if one imposes the transversality of self-intersection set of Lagrangian immersions as in \cite{akaho-joyce}.
This is because the product immersion $L_1 \times L_2$ \emph{never} satisfies the transversality
on its self-intersection set even when $L_i$ are assumed to be transversal, unless both of them
are embedded:
The self-intersection set of $L_1 \times L_2$ is given by
$$
(\Self(L_1) \times L_2) \bigcup (L_1 \times \Self(L_2))
$$
which is a continuum. However this self-intersection set is of \emph{clean intersection}. Here crucially enters
Fukaya's enlarged category \cite{fukaya:immersed} of immersed Lagrangian Floer theory since it 
is monoidal under the operation of taking transversal intersection of two immersed Lagrangians in
his sense \cite{fukaya:immersed}, but not in the sense of \cite{akaho-joyce}. 
This is the reason why one needs to enhance the immersed Lagrangian Floer theory of Akaho-Joyce
\cite{akaho-joyce} to Fukaya's \cite{fukaya:immersed}. From now on, unless otherwise
explicitly said, we will always mean an immersed Lagrangian one with clean self-intersections
in the sense of \cite{fukaya:immersed}.

This is precisely what we need in our previous paper \cite{oh-suen:lag-multi-section}
for the purpose of extending the aforementioned  \emph{microlocalization functor}
to the immersed Lagrangian Floer theory. An explicit outcome of these considerations is an
extension of Nadler's generation result from \cite{Nadler}.

The following is the main generation result which is 
the counterpart of \cite[Theorem 4.1.1]{Nadler}.

\begin{thm}\label{thm:generation} Let $\Fuk (T^*M)$ be the Fukaya category over the Novikov field generated by
weakly unobstructed immersed Lagrangian branes with clear self-intersections, and let $L$ be an object of 
the subcategory
$$
\Fuk^0(T^*M) \subset \Fuk(T^*M).
$$
The Yoneda module
$$
\mathcal{Y}_L: = hom_{\Fuk(T^*M)}(\cdot, L)
$$
can be expressed as a twisted complex with terms $hom_{\Fuk(T^*M)}(\alpha_M(\cdot), L_{\tau_{\mathfrak a}*})$.
\end{thm}

The proof of this theorem is outlined in 
\cite[Appendix A]{oh-suen:lag-multi-section}. The main purpose of the present paper is
to provide full details of the proof as promised therein.

\section{Immersed Lagrangian branes with clean self-intersections}

In this section, we briefly recall Fukaya's setting of immersed Lagrangian Floer theory
developed in \cite{fukaya:immersed}. Three key points of Fukaya's treatment therein,
which are different from that of Akaho-Joyce's \cite{akaho-joyce}, lies in the following
points:
\begin{enumerate}
\item Immersed Lagrangian submanifolds are allowed to have \emph{clean self-intersections}.
\item It uses the de Rham model and $\bb{Z}_2$ local systems for the underling vector spaces.
\item It also constructs a full Fukaya category generated by such Lagrangian immersions and
relevant functor categories and bimodules employing the calculus of Lagrangian correspondences
\cite{WW:functoriality,WW:composition}.
\end{enumerate}
The point (1) is needed for the relevant Fukaya category to be monoidal,
because it is inevitable to consider the product operation when one tries to define the 
K\"unneth-type functors.

\subsection{Immersed Lagrangian submanifolds with clean self-intersections}

Let $(X,\omega)$ be a symplectic manifold of dimension $2n$. We assume it is either
compact or tame, i.e. it admits a Riemannian metric of bounded geometry $g$ and 
a $(\omega,\mathfrak T)$-tame almost complex structure $J$ such that the metric $g_J: = \omega(\cdot, J \cdot)$ satisfies
$$
\frac1C g \leq g_J \leq C g
$$
for some $C > 0$, where $\mathfrak T$ is the quasi-isometry class of the metric $g$.  
We assume that $J$ is at least $C^3$ $(\omega,\mathfrak T)$-tame.
(See \cite{choi-oh:tameJ}  for the definition of the $(\omega,\mathfrak T)$-tameness and for the 
reason for the necessity of the condition $C^3$-tameness.)

We recall the precise definition of immersed Lagrangian submanifolds used from
\cite{fukaya:immersed} with slight variation of notations.

\begin{defn}[Definition 3.1 \cite{fukaya:immersed}] An \emph{immersed Lagrangian submanifold} $\bb{L}$
of $(X,\omega)$ is a pair $(\tildeL,i_L)$ where $\tildeL$ is a smooth manifold of dimension $n$
such that $i_L$ is an immersion. We denote by $L$ its image and $\bb{L} = (\tildeL,i_L)$.
\begin{enumerate}
\item We say $\bb{L}$ has \emph{clean self-intersection} if the following holds:
\begin{enumerate}
\item The fiber product
$$
\tildeL \times_X \tildeL = \{(p,q) \in \tildeL \times \tildeL \mid i_L(p) = i_L(q)\}
$$
is a smooth submanifold of $\tildeL \times \tildeL$.
\item For each $(p,q) \in \tildeL \times_X \tildeL$, we have
$$
T_{(p,q)}(\tildeL \times_X \tildeL) = \{(V,W) \in T_p \tildeL \times T_q \tildeL
\mid d_pi_L(V) = d_qi_L(W)\}.
$$
\end{enumerate}
\end{enumerate}
\end{defn}

Following \cite{fukaya:immersed}, we put
$$
L(+) = \tildeL \times_X \tildeL.
$$
Then we have the decomposition
\be\label{eq:L+=}
L(+) \cong \tildeL \sqcup \coprod_{a \in \cA_s(L)} L(a)
\ee
where we identify $\tildeL \cong L(+) \cap \Delta_X$ and 
$$
L(+) \setminus L(+) \cap \Delta_X =  \coprod_{a \in \cA_s(L)} L(a)
$$
where each $L(a)$ is a connected component of $L(+) \setminus L(+) \cap \Delta_X$.

\begin{defn} We call $\tildeL \subset L(+)$ the diagonal component and each $L(a)$
a switching component. We index them by the set
$$
\cA(L) = \{o\} \cup \cA_s(L), \quad L(o): = \tildeL.
$$
We say $\cL$ has  \emph{transversal self-intersection} if $L$ has a clean self-intersection and all
switching components are zero dimensional.
\end{defn}

With these definitions,  we state the following key lemma from \cite{fukaya:immersed}.

\begin{lem}[Lemma 4.6 \cite{fukaya:immersed}]\label{lem:fiber-product} Let $L_1$, $L_{12}$ be immersed 
submanifold of $X_1$ and $-X_1 \times X_2$ respectively. Assume that
$L_1 = (\tildeL_1, i_{L_1})$ has clean transformation by $L_{12} = (\tildeL_{12},i_{L_{12}})$
and denote by $L_2 = (\tildeL_2,i_{L_2})$ the associated geometric transformation. Then 
the fiber product
$$
(\tildeL_1 \times \tildeL_2) \times_{X_1 \times X_2}(\tildeL_{12})
$$
is diffeomorphic to $\tildeL_1 \times_{X_2} \tildeL_2$.
\end{lem}

An immediate corollary of this applied to $L_{12} = \Delta_X$ in $-X \times X$ is
the following.
\begin{cor}\label{cor:correspondence} 
Let $P$ and $L$ be Lagrangian immersions with clean self-intersections.
Then 
$$
(\widetilde P \times \tildeL) \times_{X \times X}(\Delta_X)
\cong \widetilde P \times_X \tildeL.
$$
\end{cor}
\begin{proof} We consider the submanifold $L \subset X$
and take the composition 
$$
L \circ \Delta_X = L \times_{X}\Delta_X = L( = : L_2).
$$
 Now we apply Lemma \ref{lem:fiber-product} to 
$$
L_1 = P, \, L_{12} = \Delta_X
$$
we obtain
$$
(\widetilde P \times \tildeL) \times_{X \times X}(\Delta_X)
\cong \widetilde P \times_X \tildeL.
$$
\end{proof}

\subsection{Brane data}

\subsubsection{Relative spin structures}

As usual, we consider the relative spin structure for the pair $(X,L)$.  More precisely, we 
start with the following definition.

\begin{defn}[Relative triangulation] Let $\bL=(\widetilde L,i_L)$ be an immersed
Lagrangian submanifold of clean self-intersection. We call a triangulation of the pair $(X,L)$
a \emph{relative triangulation} if it  satisfies the following:
\begin{enumerate}
\item Each component $L(a) \subset L$ is a subcomplex of $L$.
\item  $L$ is a subcomplex of $X$.
\end{enumerate}
\end{defn}

Then we recall the following definition from \cite{fukaya:immersed}.

\begin{defn} Let $\bL = (\widetilde L,i_L)$ be as above and assume that a
relative triangulation of $(X,L)$ is given. A \emph{relative spin structure} of $\bL$ in $X$
consists of the following data:
\begin{enumerate}
\item A real vector bundle $V$ on the 3-skeleton $X_{[3]}$ of $X$.
\item A spin structure $\sigma$ of the bundle $i_L^*V \oplus T\widetilde L$ on the 3-skeleton 
$\widetilde L_{[3]}$ of $\widetilde L$.
\item The map  $i_L: \widetilde L \to L$ is a simplicial map with respect to the given triangulations of
$\widetilde L$ and $L \subset X$.
\end{enumerate}
We put $[\text{\rm st}]= w^2(V) \in H^2(X;\bZ_2)$ and call it the \emph{background class}.
\end{defn}

\subsubsection{$\bZ_2$-local systems}

Because of the usage of clean intersections which is indeed essential to study the K\"unneth-type
functor in the immersed Lagrangian Floer theory as shown in \cite{fukaya:immersed}, we need to
equip a $\bZ_2$-local system \emph{on the switching components} which is denoted by
\be\label{eq:Z2localsystem}
\Theta_a^-
\ee
therein. (See \cite[Section 3.1]{fukaya:immersed} for the details.)

\subsubsection{Grading}

Now we specialize to the setting cotangent bundles.
	Recall that a \emph{grading} on an oriented Lagrangian immersion $i:\til{L}\to Y$ is defined to be a choice of smooth function $\theta_{\bb{L}}:\til{L}\to\bb{R}$ so that
	$$\iota^*\Omega=e^{\sqrt{-1}\theta_{\bb{L}}}{\vol}_{\til{L}},$$
	where ${\vol}_{\til{L}}$ is a volume form on $\til{L}$ with $|{\vol}_{\til{L}}|=1$ (with respective to the induced metric).
	
	Let $\bb{L}_1,\bb{L}_2$ be two graded Lagrangian immersions so that their images $L_1,L_2$ intersect transversally away from their self-intersection points. The \emph{degree} of an intersection point $p\in L_1\cap L_2$ is defined as follows. As $Y\cong\bb{C}^n$ as Calabi-Yau manifolds, there is a $U(n)$ matrix $\psi_p:\bb{C}^n\to \bb{C}^n$ so that
	\begin{align*}
		\psi_p(T_pL_1)=&\,\{(x_1,\dots,x_n)\in\bb{C}^n:x_i\in\bb{R}\},\\
		\psi_p(T_pL_2)=&\,\{(e^{\sqrt{-1}\theta_1}x_1,\dots,e^{\sqrt{-1}\theta_n}x_n)\in\bb{C}^n:x_i\in\bb{R}\},
	\end{align*}
	for some $\theta_i\in(0,\pi)$. We call these angles the \emph{angles of intersection at $p$}. The degree of $p$ from $\bb{L}_1$ to $\bb{L}_2$ is defined to be
	$$
	\deg_{\bb{L}_1,\bb{L}_2}(p):=\frac{1}{\pi}(\theta_1+\cdots+\theta_n+\theta_{\bb{L}_1}(p)-\theta_{\bb{L}_2}(p))\in\bb{Z}.
	$$
	We have $\deg_{\bb{L}_1,\bb{L}_2}(p)+\deg_{\bb{L}_2,\bb{L}_1}(p)=n$.

\section{Immersed Lagrangian filtered Fukaya category}

For a given local system $\Theta$, we denote by $\Omega(X;\Theta)$ the $\bR$ vector spacer of smooth
differential forms on $X$ with local coefficient $\Theta$.

\begin{defn}\label{defn:CFL} Let $R$ be $\bR$ or $\bC$. We consider $\Lambda_0^R$ module
\begin{eqnarray}\label{eq:complex}
CF(L;\Lambda_0^R) &: = & \Omega(L(+); \Theta) \widehat \otimes_R \Lambda_0^R \nonumber\\
& = & \left(\Omega(\tildeL) \widehat \otimes_R \Lambda_0^R\right) 
\bigoplus \left(\Omega(L(a); \Theta_a^-) \widehat\otimes_R \Lambda_0^R\right).
\end{eqnarray}
Here $\widehat \otimes$ is the $T$-adic completion of the algebraic tensor product $\otimes$.
\end{defn}

The following theorem constructs a filtered $A_\infty$ algebra 
$$
\left(CF(L;\Lambda_0^R),\{\mathfrak{m}_k\}_{k \geq 0}\right).
$$
\begin{thm}[Theorem 3.14 \cite{fukaya:immersed}] Let $L$ be as above. Suppose that $L$ has
clean self-intersection. Then there is a natural procedure of associating the structure of 
a filtered $A_\infty$ algebra on the completed free graded $\Lambda_0^R$ module $CF(L;\Lambda_0^R)$.
It is unital and is $G$-gapped for a discrete submonoid $G$.
\end{thm}

We refer readers to \cite{fukaya:immersed} for complete details of the relevant
constructions needed for the proof of this theorem, which generalizes the definition from 
\cite{akaho-joyce} to the immersed Lagrangian submanifolds with clean self-intersections.
Here we just summarize his construction to the level of our usage in our proving 
the Nadler's generation result in this enlarged setting of  immersed Lagrangian Floer theory
in the cotangent bundle later.

\subsection{Moduli spaces and evaluation maps}

In this section, we summarize the definitions of the relevant moduli spaces of 
pseuodholomorphic maps and their stable map compactification from \cite[Section 3]{fukaya:immersed}.

\begin{defn}[Definition 3.17 \cite{fukaya:immersed}]\label{defn:moduli-space} Equip $(X,\omega)$ with a compatible almost
complex structure. Let $E > 0$. We define the set 
$\stackrel{\circ}{\widetilde \cM}(L;\vec a;E)$ to be the set of all quadruples $(\Sigma,u,\vec z,\gamma)$ 
satisfying the following properties:
\begin{enumerate}
\item The space $\Sigma$ is a connected bordered stable curve of genus zero
which is a union of \emph{a disk} and a finite number of trees  of sphere components attached to
the interior of the disk. 
\item The map $u$ is $J$-holomorphic.
\item We put $\vec z = (z_0,\cdots, z_k) \subset \partial \Sigma$ a mutually distinct collection of 
marked points ordered counterclockwise on $\partial \Sigma$.
\item There is a lifting $\gamma: \del \Sigma \setminus \{z_0,\cdots,z_k\}$ of the restriction map
$$
u|_{\del \Sigma\setminus \{z_0,\cdots,z_k\}}: \del \Sigma\setminus \{z_0,\cdots,z_k\} 
\to L \setminus \Self(L).
$$
\item For each $0 \leq i \leq k$, we put
\be\label{eq:switching-bdy}
\left(\lim_{z \nearrow z_i}(\gamma(z)), \lim_{z \searrow z_i}(\gamma(z)\right) \in L(a_i).
\ee
\item We have
\be
\int_{\Sigma} u^*\omega \leq E.
\ee
\item{(Stability)} The set $\text{\rm Aut}((\Sigma,\vec z),u,\gamma)$ 
of automorphisms of $((\Sigma,\vec z),u)$
is finite.
\end{enumerate}
We denote (temporarily) by $\Aut(\Sigma,\vec z), u, \gamma)$ the set of automorphisms for the 
triple $((\Sigma,\vec z), u, \gamma)$.
\end{defn}

It is worthwhile to notice that the condition (4) above implies that 
$\gamma \in \Aut(\Sigma,\vec z), u, \gamma)$ extends continuously to $z_i$ if 
$L(a_i)$ is the diagonal component while it does not if $L(a_i)$ is a switching component.

We recall any element  $\varphi:\Sigma \to \Sigma$  in $\text{\rm Aut}((\Sigma,\vec z),u)$
satisfies
$$
u \circ \varphi = u, \, \varphi(z_i) = z_i
$$
by definition.
The following shows that  Condition (4) does not affect the totality of automorphisms 
of the stable map $((\Sigma, \vec z),u)$.

\begin{lem} Any automorphism $\varphi \in \text{\rm Aut}((\Sigma,\vec z),u)$ automatically satisfies
$$
\gamma \circ \varphi = \gamma
$$
on $\del \Sigma \setminus \{z_0,\cdots, z_k\}$. In this regard, we can define the automorphism group of
the quadruple $(\Sigma,u, \vec z, \gamma)$ to be
$$
\Aut((\Sigma,\vec z), u, \gamma): = \text{\rm Aut}((\Sigma,\vec z),u).
$$
\end{lem}
\begin{proof} Let $z \in \del \Sigma \setminus \{z_0,\cdots, z_k\}$. By definition of $\gamma$, we have
\be\label{eq:iLgamma}
(i_L \circ \gamma)(z) = u(z)
\ee
for all $z \in \del \Sigma \setminus \{z_0,\cdots, z_k\}$. Now let $\varphi \in \text{\rm Aut}(\Sigma,u,\vec z)$.
Then we have $u(\varphi(z)) = u(z)$ for all $z \in \del \Sigma$ in particular. Then by combining the two, we obtain
$$
i_L(\gamma(z)) = u(\varphi(z))
$$
for all $z \in \del \Sigma \setminus \{z_0,\cdots, z_k\}$. 
Since $\varphi(z)$ is also contained in $\del \Sigma \setminus \{z_0,\cdots, z_k\}$, we have
$$
i_L(\gamma(\varphi(z))) = u(\varphi(z))
$$
by applying \eqref{eq:iLgamma} to the point $\varphi(z) \in \del \Sigma \setminus \{z_0,\cdots, z_k\}$.
Combining the last two equalities, we derive
\be\label{eq:iLgammaz}
i_L(\gamma(z)) = i_L(\gamma(\varphi(z))) \in L \setminus \text{\rm Self}(L)
\ee
for all $z \in \del \Sigma \setminus \{z_0,\cdots, z_k\}$. Write $p: = i_L(\gamma(z)) = i_L(\gamma(\varphi(z)))$.
Consider a sequence $x_i^+ \to z$ with $x_i^+ \in \Sigma \setminus \del \Sigma$ and
$$
u(x_i^+) \to i_L(\gamma(z)).
$$
Then we have 
$$
i_L(\gamma(z)) = \lim_i u(\varphi(x_i^+)) \to u(z) = u(\varphi(z)) = \lim_i u(\varphi(x_i^+)) = u(\varphi(z)))
= i_L(\gamma(z))
$$
by continuity of the map $u$ and $\varphi$ up to the boundary. Combining this equality with \eqref{eq:iLgammaz}, we 
obtain
$$
i_L(\gamma(\varphi(z))) = i_L(\gamma(z)) \in L \setminus \text{\rm Self}(L).
$$
by Definition \ref{defn:moduli-space} (4). Since the map $i_L: \widetilde L \to L$ is
 injective on $i_L^{-1}(L \setminus \text{\rm Self}(L))$, we conclude 
$\gamma(\varphi(z)) = \gamma(z)$ for all $z \in \del \Sigma \setminus \{z_0, \cdots, z_k\}$.
This finishes the proof.
\end{proof}

We now denote by
$$
\stackrel{\circ}{\cM}(L;\vec a;E)
$$
the set of isomorphism classes of $\stackrel{\circ}{\widetilde{\cM}}(L;\vec a;E)$. 
For each \emph{switching component} $L(a)$ 
of $\text{\rm Self}(L) \subset  L$ with $a \in \cA_s(L)$ and $E > 0$, 
we associate the moduli spaces and the evaluation maps
$$
\ev = (\ev_0, \cdots, \ev_k): \stackrel{\circ}{\cM}(L;\vec a;E) \to \prod_{i=0}^k L(a_i)
$$
by
\be\label{eq:switching-evaluation}
\ev_i(u;\vec z;\gamma) = \left(\lim_{z \nearrow z_i}(\gamma(z)), \lim_{z \searrow z_i}(\gamma(z))\right).
\ee
We then denote by ${\cM}(L;\vec a;E) \to \prod_{i=0}^k L(a_i)$ the 
compactification of $\stackrel{\circ}{\cM}(L;\vec a;E) \to \prod_{i=0}^k L(a_i)$ 
described in detail in \cite{fukaya:immersed}. Then we have
the following decomposition
$$
\cM(L;\vec a;E) = \coprod_{\hat\Gamma\in \mathscr{TR}_{E;\vec a}} \stackrel{\circ}{\cM}(L;\hat\Gamma)
$$
where the coproduct is taken over some indexing set
$\mathscr{TR}_{E;\vec a}$.
(See \cite[Definition 3.20]{fukaya:immersed} for the precise definition of $\mathscr{TR}_{E;\vec a}$
and the explanation of this decomposition.)
Finally we put
$$
\cM_{k+1}(L;E): = \coprod_{\vec a \in (\cA(L))^{k+1}}\cM(L;\vec a;E).
$$
Now we compactify the moduli spaces $\stackrel{\circ}{\cM}(L;\hat\Gamma)$ for each given $E > 0$
and in turn $\cM(L;\vec a;E)$ and $\cM_{k+1}(L;E)$ in order. We then equip them with stable map
topology. The following is proved by Fukaya.

\begin{thm}[Theorem 3.23 \cite{fukaya:immersed}] The spaces $\cM(L;\vec a;E)$ and $\cM_{k+1}(L;E)$ 
are compact and Hausdorff. 
\end{thm}

\subsection{Filtered $A_\infty$ category of immersed Lagrangian Floer theory}

We fix a background class $[st] \in H^2(X,\bZ_2)$ and a triangulation of $X$.
Then we fix a real vector bundle $V$ on the 3-skeleton $X_{[3]}$.  Let
$\mathfrak{Fuk}((X,\omega);V)$ be the filtered Fukaya category and by 
$$
\mathfrak{F}(X) : = \left(\mathfrak{Fuk}((X,\omega);V)\right)^c
$$
its \emph{stricitification} in the sense of Definition \ref{defn:strictification}:

\begin{defn}[Objects] The set of objects $\mathfrak{Ob}(\mathfrak F(X))$ consists of pairs
$$
(\bL,\sigma, \Theta^-, b)
$$
where we decorate $\bL$ by
\begin{enumerate}
\item  $\bL= (\widetilde L,i_L)$ is an immersed Lagrangian submanifold,
\item $\sigma$ an $[st]$-relative spin structure of $(X,\bL)$,
\item  $\Theta^- = \{\Theta_a^-\}$ a collection of $\bZ_2$ local system on 
the switching components when $\Self(L)$ is not transversal but clean,
\item $b$ is a (weak) bounding cochain.
\end{enumerate}
\end{defn}

For the simplicity of notation, we will just write $\bL$ for $(\bL,\sigma,\Theta^-,b)$, 
provided there is no danger of confusion. Then we write
$$
hom_{\mathfrak F(X)}(\bL,\bP): = CF(\bL,\bP;\Lambda^R)
$$
as the morphism space of $\mathfrak F(X)$. \emph{These matters being clearly
mentioned, we will denote the strictification $\mathfrak F(M)$ of $\mathfrak{Fuk}((X,\omega);V)$
also by $\mathfrak{Fuk}((X,\omega);V)$  for the simplicity of notations henceforth, as long as there is no danger of confusion.}

Next we need to compare the Floer complex
$$
CF(\bL_1 \times \bL_2,\bP_1 \times \bP_2)
$$
and the tensor product
$$
CF(\bL_1,\bP_1) \widehat \otimes CF(\bL_2,\bP_2).
$$
We refer readers to \cite{amorim:kuenneth} for the definition of
the tensor product of two Fukaya algebras as in the current context.

Recall the definition
$$
CF(L;\Lambda_0^R) 
= \left(\Omega(\tildeL) \widehat \otimes_R \Lambda_0^R\right) 
\bigoplus \left(\Omega(L(a); \Theta_a^-) \widehat\otimes_R \Lambda_0^R\right)
$$ 
from Definition \ref{defn:CFL}.  From this and the natural isomorphism
\be\label{eq:monoidal}
\Omega(X \times Y;\Lambda_0^R) \cong \Omega(X;\Lambda_0^R) \otimes \Omega(Y;\Lambda_0^R)
\ee
in general, we will have the isomorphism \eqref{eq:product-isomorphism} as a $\Lambda_0^R$ module. 

However we need to prove that the two are quasi-isomorphic as an $A_\infty$ algebra. 
In this regard, the following is proved by Amorim \cite[Theorem 1.1]{amorim:kuenneth}
(embedded case) and by Fukaya \cite[Theorem 15.15] {fukaya:immersed} (immersed case)
 in the context of compact Lagrangian submanifolds. The proofs thereof can  be adapted to  the context of  cylindrical  immersed Lagrangians in 
 the cotangent bundle.
 
\begin{thm}[Theorem 15.15 \cite{fukaya:immersed}]\label{thm:product-quasiisomorphism}
Let $\bL_1$ and $\bL_2$ be immersed Lagrangians. Then there exists a filtered
$A_\infty$ quasi-isomorphism 
$$
CF(\bL_1,\bP_1) \widehat \otimes CF(\bL_2,\bP_2) \to CF(\bL_1 \times \bL_2,\bP_1 \times \bP_2).
$$
\end{thm}

Referring readers to \cite{amorim:kuenneth}, \cite{fukaya:immersed} for the full details of the proof of 
Theorem \ref{thm:product-quasiisomorphism}, here we just mention that one needs to relate the evaluation map
$$
\ev_i^1 \times \ev_i^2:\cM_{k+1}(J_1,\bL_1; \beta_1) \times \cM_{k+1}(J_2,\bL_2;  \beta_2) 
\to \widetilde L_1 \times \widetilde L_2
$$
with the evaluation map
$$
\Ev_i: \cM_{k+1}(J_1 \times J_2,\bL_1 \times \bL_2; \beta_1 \times \beta_2) \to \widetilde L_1 \times \widetilde L_2
$$
and to make certain construction of Kuranishi structures on the moduli space
$\cM_{k+1}(J_1 \times J_2,\bL_1 \times \bL_2; \beta_1 \times \beta_2)$ suitably 
adapted to the projection maps
$$
\Pi_j: \cM_{k+1}(J_j,\bL_j; \beta_j), \quad j = 1, \, 2.
$$
(See \cite[Section 3]{amorim:kuenneth}, \cite[Section 5.2]{fukaya:immersed} for details.)

\section{Nadler's generation result for the immersed Lagrangians}
\label{sec:nadler}

In this section, we provide some more explanation on how Nadler's proof
in \cite[Section 4]{Nadler} applies to the case of \emph{tautologically unobstructed} 
compact immersed Lagrangians in the cotangent bundle $T^*M$ for general 
manifold $M$ that has the cylindrical structure at infinity, i.e., that there exists a compact
subset $C \subset M$ such that we have
$$
M \setminus K \cong [0, \infty) \times \partial_\infty M
$$
as a Riemannian manifold, where $\partial_\infty M$ is a closed manifold and 
its metric $g = ds^2 + s \, g_{\partial N}$. In the case of our main interest in \cite{oh-suen:lag-multi-section}, 
$M = M_{\bb{R}}$ and $\partial_\infty M \cong S^{n-1}$.

We start with giving the definition of \emph{tautologically unobstructed}  Lagrangian 
immersions.
\begin{defn}\label{defn:tautologically-unobstructed} We say that an immersed Lagrangian brane
$\bL: = ((i_L: \widetilde L \to L), \sigma, \Theta^-)$ is  \emph{tautologically unobstructed}
if $\mathfrak{m}_0^\bL(1) = 0$.
\end{defn}
For any such brane, the zero element $b = 0$ is a (strict) Maurer-Cartan element, i.e., 
the \emph{un-deformed} $\mathfrak m_1$ itself satisfies $\mathfrak m_1^2 = 0$ 

In the rest of the present paper, we will focus on the case of tautologically unobstructed 
(immersed) Lagrangian branes which are needed for the main purpose of \cite{oh-suen:lag-multi-section}.

The following is the main generation result of the present paper, which is announced in
\cite{oh-suen:lag-multi-section} and  which is the counterpart of \cite[Theorem 4.1.1]{Nadler}
in the framework of the current enlarged immersed Lagrangian Fukaya category.
(For the definition of $\Fuk^0(T^*M)$ below, we refer readers to Definition \ref{defn:strictification}, 
the decomposition \eqref{eq:L-decompose} and Proposition \ref{prop:curvature-free} in Appendix.)

\begin{thm}\label{thm:generation} Let $\Fuk(T^*M)$ be the Fukaya category generated by weakly unobstructed
immersed Lagrangian branes and let $\bL$ be an object of $\Fuk^0(T^*M) \subset \Fuk(T^*M)$. 
The Yoneda module
$$
\mathcal{Y}_\bL: = hom_{\Fuk(T^*M)}(\cdot, \bL)
$$
can be expressed as a twisted complex with terms 
$hom_{\Fuk(T^*M)}(\alpha_M(\cdot), L_{\tau_{\mathfrak a}*})$.
\end{thm}

The main argument in Nadler's proof consist of 4 parts (See \cite[Section 4.5]{Nadler}.)
Nadler's proof:
\begin{enumerate}
\item Description of product branes in $\Fuk^0(T^*M_0 \times T^*M_1)$,
\item Triangulation of diagonal,
\item Moving the diagonal,
\item Combining them all.
\end{enumerate}

\medskip

We explain how we can adapt each of these Nadler's argument to the current immersed setting
in each part of the above 4 parts. Since the product $L_1 \times L_2$ of two
immersed Lagrangians cannot have transverse self-intersections but only clean
self-intersections of higher dimension, we need to employ Fukaya's framework of 
immersed Lagrangian Floer theory whose objects consist of immersed
Lagrangian submanifolds with clean self-intersections \cite{fukaya:immersed}.

\subsection{Product branes in $\Fuk(T^*M_0 \times T^*M_1)$}

Let $M_i$ for $i=1,\,2$ be either compact or tame manifolds and consider their
cotangent bundles $T^*M_i$. We consider the product map on the set of branes of immersed Lagrangian submanifolds
$$
\bb{L}_i = (\widetilde L_i, i_{L_i}, \Theta_i), \quad i = 1, \, 2.
$$
\begin{lem}[Compare with Lemma 4.2.1 \cite{Nadler}] For any test objects $\bL_0$, $\bP_0$ of 
$\Fuk(T^*M_0)$ and $\Fuk(T^*M_1)$, there is a functorial (in each object) quasi-isomorphism of 
Floer complexes
\be\label{eq:product-isomorphism}
hom_{\mathfrak F(T^*M_0 \times T^*M_1)}(\bL_0 \times \bL_1,\bP_0 \times \bP_1)
\cong hom_{\mathfrak F(T^*M_0)}(\bL_0,\bP_0)
\times hom_{\mathfrak F(T^*M_1)}(\bL_1, \bP_1).
\ee
\end{lem}
\begin{proof} This lemma follows from Theorem \ref{thm:product-quasiisomorphism}.
\end{proof}

\subsection{Diagonal correspondence and its associated $A_\infty$ functor}

Fukaya also extended Theorem \ref{thm:product-quasiisomorphism} to the categorial level.
The following statement 
is a variant of \cite[Theorem 16.8]{fukaya:immersed} and \cite[Proposition 4.2.2]{Nadler}:

\begin{thm}[Theorem 16.8 \cite{fukaya:immersed}] There exists a strict and unital filtered
$A_\infty$ functor
$$
\mathfrak{Fuk}(X_1) \otimes \mathfrak{Fuk}(X_2) 
\to \mathfrak{Fuk}(X_1 \times X_2).
$$ 
\end{thm}
 
Then using this general result, we can duplicate 
Nadler's argument used in his proof of \cite[Proposition 4.2.2]{Nadler} by
adapting his proof to the case of tautologically unobstructed immersed Lagrangian submanifolds 
$\bP_i$.

The following is the counterpart of \cite[Proposition 4.2.2]{Nadler}.

\begin{prop}[Compare with Theorem 5.15 \cite{fukaya:immersed},
Proposition 4.2.2 \cite{Nadler}]\label{prop:generation} For any test objects
$\bP_0,\, \bP_1$ of $\Fuk^0(T^*M)$, there is a functorial quasi-isomorphism of 
Floer complexes
$$
hom_{\Fuk(T^*M)}(\bP_0,\bP_1) \cong 
hom_{\Fuk(T^*M \times T^*M)}(L_{\Delta_M}, \bP_0 \times \alpha_M(\bP_1)).
$$
\end{prop} 

Once these are established, the current subsection is purely about the diagonal on 
$$
L_{\Delta_M}= \{(x,\xi; x, -\xi) \in T^*M \times T^*M\},
$$
and Nadler's proof of \cite[Proposition 4.2.2]{Nadler}
applies verbatim without any change even in the context of present paper.

\begin{rem} On one hand, Proposition \ref{prop:generation}
 is a special case restricted to the case of tautologically 
unobstructed immersed Lagrangians of \cite[Theorem 5.15]{fukaya:immersed},
and on the other hand is a generalization of \cite[Proposition 4.2.2]{Nadler} to the case 
of immersed Lagrangians. Indeed Nadler's proof is a special case of Fukaya's general result on the
$A_\infty$ tri-functor associated to Lagrangian correspondence \cite[Theorem 5.15]{fukaya:immersed} 
applied to the case of cotangent bundles. 

Here we just mention that the main geometric comparison result on the relevant 
moduli spaces of pseudoholomorphic curves is based on Corollary \ref{cor:correspondence}.
(See the proof of \cite[Proposition 4.2.2]{Nadler}.)
\end{rem}

\subsection{Nadler's resolution of the diagonal by the standard branes}

This is an adaptation of \cite[Section 4.3]{Nadler} and the following is the main
lemma to be proved therein.

\begin{prop}[Lemma 4.3.1 \cite{Nadler}] The dual brane $\alpha_M(L_{\Delta_M})$ can be expressed as a
twisted complex of the standard branes $L_{\Delta_{\tau_{\mathfrak a}*}}$.
\end{prop}
\begin{proof}[Outline of the proof]
We briefly summarize the main arguments used in the proof. Fix a triangulation $\cT = \{\tau_{\mathfrak a}\}$ of $M$,
and consider the $\cT$-constructible dg derived category $Sh_{\cT}(M)$.
Consider the inclusion $j_{\mathfrak a}: \tau_{\mathfrak a} \hookrightarrow M$, and
the corresponding standard sheaves $j_{\mathfrak{a}*}\bC_{\tau_{\mathfrak a}}$.

Then we summarize a few result established by Nadler in \cite{Nadler}.

\begin{lem}[Lemma 2.3.1 \cite{Nadler}]
Any object of $Sh_\cT(M)$, in particular the constant sheaf $\bC_M$,
as an iterated cone of maps among the standard sheaves $j_{\mathfrak{a}*}\bC_{\tau_{\mathfrak a}}$.
\end{lem}

We consider the inclusion
$$
d_{\mathfrak{a}}: \Delta_{\tau_{\mathfrak{a}}} \hookrightarrow \Delta_M.
$$
\begin{lem}[Section 4.3 \cite{Nadler}] The following microlocalizations hold:
\begin{enumerate}
\item The diagonal brane $L_{\Delta_M}$ is the standard object of $\Fuk(T^*M \times T^*M)$
obtained as 
$$
L_{\Delta_M} \cong \mu_{M \times M}(\bC_{\Delta_M}).
$$
\item The standard object $L_{\Delta_{\tau_{\mathfrak{a}*}}}$ is obtained as 
$$
L_{\Delta_{\tau_{\mathfrak a}*}} \cong \mu_{M \times M}(d_{\mathfrak{a}*}\bC_{\Delta_M}).
$$
\end{enumerate}
\end{lem}

We fix a non-negative function 
$$
m_{\mathfrak a}: M \to \bR
$$
that vanishes precisely on $\del \tau_{\mathfrak a} \subset M$,
and consider the function $f_{\mathfrak a} : \tau_{\mathfrak a} \to \bR$ given by
$$
f_{\mathfrak a}: = \log m_{\mathfrak a}|_{\mathfrak a}.
$$
Then the underlying Lagrangian of $L_{\Delta_{\tau_{\mathfrak a}}!}$ can be identified with the fiber sum
$$
T^*_{\Delta_{\tau_{\mathfrak a}}}(M \times M) + \Gamma_{p_2^*df_{\mathfrak a}} 
\subset T^*(M \times M),
$$
where $\Gamma_{p_2^*df_{\mathfrak a}}$ denotes the graph of the pullback $p_2^*df_{\mathfrak a}$
via the projection $p_2: M \times M \to M$. This ends the summary of the proof of the proposition.
\end{proof}

\subsection{Moving the diagonal}

The following is the main ingredient in the proof of Theorem \ref{thm:generation}.
By going through the work of Nadler \cite{Nadler}, we find that for the tautologically unobstructed category, the microlocalization functor is still essentially surjective.

\begin{prop}[Compare with Proposition 4.4.1 \cite{Nadler}] Let $Y = T^*M$.
For any tautologically unobstructed test objects $\bL,\bP$,
the $\Fuk(Y)$-module $\Hom_{\Fuk(Y)}(\bb{P},\bb{L})$ can be decomposed into a twisted complex of standard modules
		$$
		\Hom_{\Fuk(Y)}(\alpha_{M}(\bb{L}),L_{\{x_{\mathfrak a}\}})\otimes 
\Hom_{\Fuk(Y)}(\bb{P},L_{\tau_{\mathfrak a}*}).
		$$
\end{prop}
\begin{proof}[Outline of the proof]
		In \cite{Nadler}, Nadler showed that for a 
		  a fine enough stratification $\cu{T}=\{\tau_a\}$ gives of $M$ so that $L^{\infty}\subset\Lambda_{\cu{T}}^{\infty}$ and $\alpha_{M}:T^*M\to T^*M$ is the involution $(x,\xi)\mapsto(x,-\xi)$.
		(Nadler's proof relies on the result that the diagonal of $T^*M^- \times T^*M$ can be
		resolved by the standard branes via the argument of utilizing the Lagrangian correspondence functor, 
which holds the case independent of the coefficient field.)
		The only difference here is that the (finite-dimensional) vector spaces $\Hom_{\Fuk(Y)}(\alpha_{M}(\bL),L_{\{x_a\}}),\Hom_{\Fuk(Y)}(\bP,L_{\tau_a*})$ and morphisms are now defined 
over $\bb{K}$ when $\bL,\, \bP$ are tautologically unobstructed immersed Lagrangian branes.
\end{proof}

\subsection{Wrap up of the proof of Theorem \ref{thm:generation}}

We adapt the arguments of \cite[Section 4.5]{Nadler} here. We fix an object $\bL$ of $\Fuk^0(T^*M)$. We need to show that 
for any text object $\bP$ of $\Fuk^0(T^*M)$, one can functorially express the Floer
complex $hom_{\Fuk(T^*M)}(\bP,\bL)$ as a twisted complex of the Floer complexes $hom_{\Fuk(T^*M)}(\bP,L_{S*})$ of 
standard branes $L_{S*} \hookrightarrow T^*M$ associated to submanifolds $S \hookrightarrow M$.

First by Proposition \ref{prop:generation}, there is a functorial quasi-isomorphism 
$$
hom_{\Fuk(T^*M)}(\bP,\bL) \cong hom_{\Fuk(T^*M \times T^*M)}(L_{\Delta_M}, \bL \times \alpha_M(\bP)).
$$
We can rewrite this into
$$
hom_{\Fuk(T^*M)}(\bP,\bL) \cong hom_{\Fuk(T^*M \times T^*M)}(\alpha_M(\bL) \times \bP, L_{\Delta_M})
$$
by applying the anti-equivalence $\alpha_M$.

Next, we fix a triangulation $\cT = \{\tau_{\mathfrak a}\}$ of $M$ along with the induced triangulation
$\Delta_\cT = \{\Delta_{\tau_{\mathfrak a}}\}$ of $\Delta_M \subset M \times M$. Then we can express 
$\alpha_M(L_{\Delta_M})$ as a twisted complex of the standard branes 
$L_{\Delta_{\tau_{\mathfrak a}}*}$.
We assume further that $\cT$ is fine enough so that $L^\infty \subset \Lambda_\cT^\infty$.
Now by Proposition \ref{prop:generation}, there is a functorial quasi-isomorphism of complexes
$$
hom_{\Fuk(T^*M \times T^*M)}(\alpha_M(\bL) \times \bP, \alpha_M(\bP))
\cong
hom_{\Fuk(T^*M)}(\alpha_M(\bL),L_{\Delta_{\tau_{\mathfrak a}*}}) \otimes
hom_{\Fuk(T^*M)}(\bP_0,L_{\Delta_{\tau_{\mathfrak a}*}}).
$$
Therefore we have expressed the Floer complex $hom_{\Fuk(T^*M)}(\bP,\bL)$ as a twisted complex
with terms the Floer complexes $hom_{\Fuk(T^*M)}(\alpha_M(\bP),L_{\Delta_{\tau_{\mathfrak a}*}})$.
This completes the proof of Theorem \ref{thm:generation}.

\appendix

\section{Filtered $A_\infty$ categories and their strictifications}
\label{sec:filtered-fukaya-category}

In this appendix, we recall the general definition of curved filtered $A_\infty$ categories 
 from \cite{fooo:book1}, \cite{fooo:anchored} and their  \emph{strictifications}: 
 We refer readers to \cite{fukaya:immersed},
 \cite[Section 7.1]{oh:kias-exposition} for detailed explanation on their constuction.

\subsection{Curved filtered $A_\infty$ categories}

\noindent Let $R$ be a commutative ring with unit. Denote by $\Lambda_0^R$ the universal Novikov ring over $R$, i.e., the set of formal sums $\sum_{i=1}^\infty a_i T^{\lambda_i}$, where $a_i \in R$
with $a_0 \neq 0$ and the maps $i \to \lambda_i$ satisfy
$$
0 \leq \lambda_1 < \cdots < \lambda_i < \lambda_{i+1} < \cdots
$$
as $i \to \infty$, \emph{unless $a_i = 0$ for all but finite $i$'s}. We then denote by
$\Lambda^R$ its field of fractions, and $\Lambda^R_+$ its maximal ideal.

\begin{rem}
\begin{enumerate}
\item
In our geometric application of the Fukaya category for the general compact symplectic manifold, $R$ will be
the polynomial ring $\bC[e,e^{-1}]$ or $\bR[e,e^{-1}]$ over a formal variable $e$ of degree 2 which encodes
the Maslov index.
\item In the context of homological mirror symmetry of Calabi--Yau manifolds, where the $A$-model is the Fukaya category
generated by Maslov index 0 Lagrangian branes, $R$ will be $\bC$.
\end{enumerate}
\end{rem}

\begin{defn}[Energy filtration]
The ring $\Lambda_0^R$ carries a natural filtration
\be\label{eq:filtration-Lambda0}
F^\lambda \Lambda_0^R = \left\{\sum_{i=1}^\infty a_i T^{\lambda_i} \in \Lambda_0^R
\, \Big|\,  a_i = 0,
\, \forall  \lambda_i < \lambda \right\}.
\ee
We call this filtration the \emph{energy filtration}.
 We define the valuation
\be
\mathfrak \nu \left(\sum_{i=1}^\infty a_i T^{\lambda_i}\right): = \lambda_1.
\ee
(Recall from above that we assume $a_0 \neq 0$.)
As usual, we set $ \nu(0): = \infty$.
\end{defn}

This filtration provides a natural non-Archimedean topology induced by the metric defined by
$$
d_{\mathfrak \nu} (\mathfrak a_1, \mathfrak a_2) : = e^{-\nu(\mathfrak a_2 - \mathfrak a_1)}.
$$
The function $\nu$ satisfies the non-Archimedean triangle inequality
$$
\mathfrak \nu (\mathfrak a_1 + \mathfrak a_2) \geq \min \{\nu(\mathfrak a_1),  \nu(\mathfrak a_2)\}.
$$
\begin{defn}[Monoid and Novikov ring]
\begin{enumerate}
\item  A \emph{discrete monoid} $G$ is a discrete subset of $\bR_{\geq 0}$
such that $0 \in G$ and satisfies that if $g_1, \, g_2 \in G$, then $g_1 + g_2 \in G$.
\item For a given discrete monoid $G \subset \bR_{\geq 0}$, we define $\Lambda_G^R$
to be the subring of the universal Novikov ring $\Lambda_0^R$  given by
\be\label{eq:Lambda-RG}
\Lambda_{G ,0}= \left\{\sum_i a_i T^{\lambda_i}  \in \Lambda_0 \,
\Big|\,  \lambda_i \in G \right\}.
\ee
\end{enumerate}
We denote by $\Lambda_G$ the Novikov field associated to $\Lambda_{G,0}$.
\end{defn}

Next, we consider a free graded module over $R$, which we denote by
$\overline C$, and
$$
C = \overline C \otimes_R \otimes \Lambda_0.
$$
\begin{hypo}\label{hypo:Cbar}
We assume
that either $C$ is a de Rham complex, or $\overline C^k$ is finitely generated for each $k$
or a combination of them.
\end{hypo}
Note that for the purpose of constructing the Fukaya category generated by 
immersed Lagrangian branes with clean self-intersections, the associated complex
is a finite (completed) direct sum of the two types. (See Definition \ref{defn:CFL}.)
The second type occurs when some components $L(a)$ of $\text{\rm Self}(L)$ are
zero-dimensional, i.e., transversal, which can be identified with the de Rham complex
over the zero dimensional manifolds.

\begin{defn}[$G$-gapped] Let $G$ be a discrete monoid and $C$ be module over 
$\Lambda_0$. 
\begin{enumerate}
\item An element $x$ of $C$ is called  \emph{$G$-gapped} if it is expressed as a sum
(not necessarily finite)
$$
x = \sum_{g \in G}T^g x_g
$$
with $x_g \in \overline C$.
\item An $\Lambda_0^R$ homomorphism $\varphi$ between $C_1, \, C_2$ is called \emph{$G$-gapped} if we can decompose  $\varphi$ into a sum (not necessarily finite)
\be\label{eq:gapped-varphi}
\varphi = \sum_{g \in G} T^g \varphi_g
\ee
for  $R$ module homomorphisms $\varphi_g: \overline C_1 \to \overline C_2$.

For a $G$-gapped homomorphism $\varphi$, we write $\overline \varphi: = \varphi_0$
for $\varphi_0$ in \eqref{eq:gapped-varphi} and call it the \emph{$R$-reduction of $\varphi$}.
\end{enumerate}
\end{defn}
We note that the expression \eqref{eq:gapped-varphi} is equivalent to the property that
$\varphi$ maps $G$-gapped elements of $C_1$ to $G$-gapped ones in $C_2$.

\begin{defn}[Definition 2.2 \cite{fukaya:immersed}]
A \emph{non-unital curved filtered $A_\infty$
category} $\mathscr C$ consist of the following: A set $\mathfrak{Ob}(\mathscr C)$, the set of objects,
and a graded completed free $\Lambda_0$ module $\mathscr(c_1,c_2)$ for each pair of
$c_1, \, c_2 \in \mathfrak{Ob}(\mathscr C)$, and the collection of operations
$$
\mathfrak m_k: \mathscr C[1](c_0,c_1) \widehat \otimes \cdots \widehat \otimes \mathscr C[1](c_{k-1},c_k)
\longrightarrow \mathscr C[1](c_0,c_k).
$$
\begin{enumerate}
\item The operator $\mathfrak m_k$ for $k \geq 0$ preserves the filtration
$$
\mathfrak m_k(\mathfrak F^\lambda(\mathscr C[1](c_0,c_1) \widehat \otimes \cdots \widehat \otimes
\mathscr C[1](c_{k-1},c_k)) \subset \mathscr F^\lambda(\mathscr C[1](c_0,c_k)).
$$
\item We have
$$
\mathfrak m_0(1) \equiv 0 \mod T^\epsilon$$
for some $\epsilon > 0$.
\item $\mathfrak m_k$ satisfies the $A_\infty$ relation.
\end{enumerate}
\end{defn}
Note in a curved  filtered $A_{\infty}$ category the operator $\mathfrak m_0$ may not be $0$.

\subsection{Strictification of curved filtered $A_\infty$ category}

\noindent Recall that in a curved  filtered $A_{\infty}$ category the operator $\mathfrak m_0$
may not be $0$. In general, non-zero `curvature' gives rise to a `BRST anomaly' i.e.,
$\mathfrak m_1 \circ \mathfrak m_1 \neq 0$.  There is a way of extracting an
\emph{anomaly-free} or \emph{unobstructed} filtered $A_\infty$ category
first by shrinking the set of objects and then by decorating each object
by a Maurer--Cartan element. We call this whole process the \emph{strictification
of a curved filtered $A_\infty$ category}.

\begin{defn}[Maurer--Cartan moduli space] Let $c \in \mathfrak{Ob}(\mathscr C)$.
\begin{enumerate}
\item We define
$$
\widetilde {\mathfrak M}(c;\Lambda_+)
$$
to be the set of $b \in \mathscr C^1 = \mathscr C[1]^0$ satisfying the following:
\begin{enumerate}
\item $b \equiv 0 \mod \Lambda_+$.
\item $b$ satisfies
\be\label{eq:b-MC}
\sum_{k=0}^\infty \mathfrak m_k(b, \ldots, b) = 0.
\ee
We call such a $b$ a \emph{bounding cochain}.
\end{enumerate}
\item We define ${\mathfrak M}(c;\Lambda_+) := \widetilde {\mathfrak M}(c;\Lambda_+)/\sim$
the set of gauge-equivalence classes, and call
$$
c \mapsto {\mathfrak M}(c;\Lambda_+); \quad \mathfrak{Ob}(\mathscr C) \longrightarrow \mathfrak{Set}
$$
the \emph{Maurer--Cartan moduli space} of $c \in \mathfrak{Ob}(\mathscr C)$.
\item By definition, there exists a scalar $\PO(c) \in R$ such that
\be
  {\mathfrak m}^{{\mathscr C}[1]}_{0}(1) = \PO(c) \cdot \id_{c}
\ee
for any $c \in \mathfrak{Ob}(\mathscr C)$
We call the valuation $\nu(\PO(c)$ the \emph{potential level} of $c$.
\end{enumerate}
\end{defn}

\begin{defn}\label{defn:weakly-unobstructed-c}
An object $c$  is \emph{weakly unobstructed} if $\widetilde {\mathfrak M}(c;\Lambda_+) \neq \emptyset$.
\end{defn}
 Note that the $A_{\infty}$ equation for $d=2$ reads
  \be
 {\mathfrak m}_{1}^{{\mathscr C}}\left( {\mathfrak m}_{1}^{{\mathfrak C}}(a)\right) + {\mathfrak m}_{2}^{{\mathscr C}}\left( {\mathfrak m}_{0}^{{\mathfrak C}}(1), a\right) + (-1)^{|a|'}  {\mathfrak m}_{2}^{{\mathscr C}}\left(a,  {\mathfrak m}_{0}^{{\mathfrak C}}(1)\right) = 0.
  \ee
 \begin{lem} The endomorphism $   {\mathfrak m}_{1}^{{\mathscr C}}$  defines a differential on the  morphism spaces ${\mathscr C}(c_0,c_1)$ provided
 $c_0$ and $c_1$ are weakly unobstructed and $\PO(c_0) = \PO(c_1)$.
 \end{lem}

We are now ready to provide the definition of the strictification $\mathscr C$.

\begin{defn}[Strictification]
\label{defn:strictification}
Suppose that  $\mathscr C$ is a non-unital curved filtered $A_\infty$ category.  The
\emph{strictification}, denoted by ${\mathscr C}^s$, of $\mathscr C$ is defined as follows:
\begin{enumerate}
\item An object of ${\mathscr C}^s$ is a pair
$$
(c,b) \in \mathfrak{Ob}(\mathscr C), \quad b \in \widetilde{\mathfrak M}(\mathscr(c,c);\Lambda_+).
$$
\item We define the morphism space ${\mathscr C}^s((c,b),(c',b'))$ to be
$$
{\mathscr C}^s((c,b),(c',b')): = {\mathscr C}(c,c').
$$
\item We define the structure maps ${\mathfrak m}_k^{\vec b}$ by
\be\label{eq:mkb}
{\mathfrak m}_k^{\vec b}(x_1, \cdots, x_k)
: = \sum_{\ell_0,\ldots, \ell_k}
{\mathfrak m}_{k + \ell_0 + \cdots + \ell_k}\left(b_0^{\ell_0}, x_1, b_1^{\ell_1}, \cdots,
b_{k-1}^{\ell_{k-1}}, x_k, b_k^{\ell_k}\right)
\ee
for given $(c_i,b_i) \in \mathfrak{Ob}({\mathscr C}^s)$ for $i = 0, \ldots, k$ and
$$
x_i \in {\mathscr C}^s((c_{i-1},b_{i-1}),(c'_i,b'_i)): = {\mathscr C}(c_{i-1},c'_i).
$$
Here we write $\vec b = (b_0,\cdots, b_k)$.
\end{enumerate}
\end{defn}

We now denote the strictification of the curved filtered $A_\infty$ category $\mathscr C$ by
$$
\mathfrak C = {\mathscr C}^s
$$
as a curvature-free filtered $A_\infty$ category. By definition each object
$X = (c,b)$ carries its potential level $\mathfrak{PO}(X): = \mathfrak{PO}(c,b)$.

We divide the collection ${\bf L}$ of objects into the disjoint union
\be\label{eq:L-decompose}
{\bf L} = \bigcup_{\lambda} {\bf L}_{\lambda},
\ee
such that ${\bf L}$ is a subset of 
$$
 \{(L_\kappa, \sigma_\kappa, b_\kappa) \mid L_\kappa \in \mathscr{L}, \,\, b_\kappa \in \mathcal M_{\rm weak}(\Omega(L_\kappa),\mathfrak b, \sigma_\kappa;\Lambda_0)\}
$$
and 
$$
 {\bf L}_{\lambda} = \{ (L, \sigma_L , b) \in {\bf L} \mid \mathfrak{PO}_{\mathfrak b}(b) = \lambda\}
$$
is the iso-level sub-collection thereof with potential level $\lambda$. (See Definition
\ref{defn:weakly-unobstructed-c}
for the definition.)

$\mathfrak{Fuk}_\lambda$ is the full subcategory of $\cL$ whose object is an element of
${\bf L}_{\lambda}$. Then if $(L,\sigma_L, b),(L',\sigma_{L'},b')$ are objects of ${\bf L}_{\lambda}$
we have $\mathfrak m_1 \circ \mathfrak m_1 =0$ on $CF((L,b),(L',b');\Lambda_0)$.
When we consider the filtered $A_{\infty}$ category $\mathfrak{Fuk}_\lambda$
we {\it re-define} $\mathfrak m_0$ so that $\mathfrak m_0 = 0$ as mentioned before.
Then the facts that $\mathfrak{PO}_{\mathfrak b}(b)$ is independent of $(L,\sigma_L, b) \in {\bf L}_{\lambda}$ and that
$\mathfrak m_0(1)$  is
proportional to the unit imply that the $A_{\infty}$ formula holds with this
re-defined $\mathfrak m_0$. We summarize the above discussion into the following:
\begin{prop}\label{prop:curvature-free}
For each fixed $\lambda$,
$\mathfrak{Fuk}_\lambda$ forms a curvature-free filtered $A_{\infty}$ category.
\end{prop}

\begin{rem}
The $A_\infty$ category $\Fuk_0  = \Fuk^0(T^*M)$
plays a crucial role in the main results of  the present paper.  The class of objects
of $\Fuk^0(T^*M)$
includes all compact exact Lagrangian branes and a priori contains more. 
We postpone further refined study of $\Fuk^0(T^*M)$, after twisting the branes by
local systems, elsewhere.
\end{rem}

	
\providecommand{\bysame}{\leavevmode\hbox to3em{\hrulefill}\thinspace}
\providecommand{\MR}{\relax\ifhmode\unskip\space\fi MR }
\providecommand{\MRhref}[2]{%
  \href{http://www.ams.org/mathscinet-getitem?mr=#1}{#2}
}
\providecommand{\href}[2]{#2}

\end{document}